\documentclass[12pt]{amsart}
\usepackage{amsmath,amsthm,amsfonts,amssymb,latexsym,enumerate,xcolor} 
\usepackage{showlabels}
\usepackage[pagebackref]{hyperref} 

\newcommand{\F}{\mathbb{F}}
\newcommand{\N}{\mathbb{N}}
\newcommand{\Q}{\mathbb{Q}}

\newcommand{\Aut}{\operatorname{Aut}}

\newcommand{\Irr}{\operatorname{Irr}}

\newcommand{\PSL}{\operatorname{PSL}}
\newcommand{\Sz}{\operatorname{Sz}}
\newcommand{\Gal}{\operatorname{Gal}}

\newtheorem{thm}{Theorem}[section]
\newtheorem{lem}[thm]{Lemma}

\newtheorem{cor}[thm]{Corollary}

\newtheorem*{thmA}{Theorem A}
\newtheorem*{conA'}{Conjecture A'}

\theoremstyle{definition}

\numberwithin{equation}{section}

\begin{document}

\title[Groups with small multiplicities of fields of values]{Groups with small multiplicities of fields of values of irreducible characters}

\author{Juan Mart\'inez}
\address{Departament de Matem\`atiques, Universitat de Val\`encia, 46100
  Burjassot, Val\`encia, Spain}
\email{Juan.Martinez-Madrid@uv.es}

\thanks{Research supported by Generalitat Valenciana CIAICO/2021/163 and CIACIF/2021/228.}

\keywords{Irreducible character, Field of values, Galois extension}

\subjclass[2020]{Primary 20C15}

\date{\today}

\begin{abstract}
In this work, we  classify all finite groups such that for every field extension $F$ of $\Q$, $F$ is the field of values of at most $3$ irreducible characters.
\end{abstract}

\maketitle

\section{Introduction}\label{Section1}

 Let $G$ be a finite group, and let $\chi$ be a character of $G$. We define the field of values of $\chi$ as

\[\Q(\chi)=\Q(\chi(g)|g \in G).\]   We also define \[f(G)=\max_{F/\mathbb{Q}}|\{\chi \in \Irr(G)|\mathbb{Q}(\chi)=F\}|.\] A.Moretó \cite{Alex} proved that the order of a group is bounded in terms of $f(G)$. This is, there exists $b : \N \rightarrow \N$ such that $|G|\leq b(f(G))$, for every finite group $G$. In that work,  it was observed  that $f(G)=1$ if and only if $G=1$. The referee of \cite{Alex} asked for the classification of finite groups $G$ with $f(G)=2$ or $3$. Our goal in this paper is to obtain this classification.

\begin{thmA}
Let $G$ be a finite group. Then

\begin{itemize}
\item[(i)] If $f(G)=2$, then $G \in \{\mathsf{C}_{2},\mathsf{C}_{3},\mathsf{C}_{4},\mathsf{D}_{10},\mathsf{A}_{4},\mathsf{F}_{21}\}$.

\item[(ii)]  If $f(G)=3$, then $G \in \{\mathsf{S}_{3},\mathsf{D}_{14},\mathsf{D}_{18},\mathsf{F}_{20},\mathsf{F}_{52}, \mathsf{A}_{5},\PSL(2,8),\Sz(8)\}$.
\end{itemize}
where $\mathsf{F}_{n}$ and $\mathsf{D}_{n}$ are the Frobenius group and the dihedral group of order $n$, respectively.  As a consequence, the best possible values for $b(2)$ and $b(3)$ are $21$ and $29.120$, respectively.
\end{thmA}

We will study the solvable case and the non-solvable case separately. In the non-solvable case, using a theorem of Navarro and Tiep \cite{Navarro-Tiep}, we will prove that the condition $f(G)\leq 3$ implies that $G$ possesses  $3$ rational characters. Then, we will use the main results of   \cite{Rossi} to restrict the structure of non-solvable groups with $f(G)\leq 3$. We will divide the solvable case in two different steps. In the first step, we classify all metabelian groups with $f(G)\leq 3$. To do this  we will use the condition $f(G)\leq 3$ to give an upper bound to the number of irreducible characters, or equivalently, an upper bound to the number of conjugacy classes. Once we have bounded the number of conjugacy classes, we will use the classification given in  \cite{VeraLopez} to finish our classification. In the  second step, we prove that if $G$ is a solvable group with $f(G)\leq 3$, then $G$ is metabelian.

Our work shows that, as expected, the bounds that are attainable from \cite{Alex} are far from best possible.  Following the proof in \cite{Alex} we can see that if $f(G)=2$ and $G$ is solvable, then  $G$ has at most $256$ conjugacy classes.  It follows from Brauer's  \cite{Brauer} bound for the order of a group in terms of its number of conjugacy classes, that  $|G|\leq 2^{2^{256}}$. We remark that, even though there are asymptotically better more recent bounds, they depend on non-explicit constants and it is not clear if they are better for groups with at most $256$ conjugacy classes.

\section{Preliminaries}\label{Section2}

In this section we present the basic results  that will be used in this work, sometimes without citing them explicitly.

\begin{lem}
Let $G$ be a finite group. If $N$ is a normal subgroup of $G$, then $f(G/N)\leq f(G)$.
\end{lem}

\begin{lem}[Lemma 3.1 of \cite{Alex}]\label{cf}
Let $G$ be a finite group and $\chi \in \Irr(G)$. Then $|\mathbb{Q}(\chi):\mathbb{Q}|\leq f(G)$. 
\end{lem}

As a consequence of this result, if $f(G)\leq 3$, then $|\mathbb{Q}(\chi):\mathbb{Q}|\leq 3$. Therefore,  $\Q(\chi)$ will be  $\Q$, a quadratic extension of $\Q$ or a cubic extension of $\Q$. We can also deduce that if $f(G)\leq 3$  and $\chi \in \Irr(G)$, then there exists $g \in G$ such that $\Q(\chi)=\Q(\chi(g))$.

\begin{lem}
Let $G$ be a group with $f(G)\leq 3$ and $\chi \in \Irr(G)$ such that $|\Q(\chi):\Q|=2$. Then $\{\psi \in \Irr(G)|\Q(\psi)=\Q(\chi)\}=\{\chi,\chi^{\sigma}\}$, where $\Gal(\Q(\chi)/\Q)=\{1,\sigma\}$.
\begin{proof}
Clearly $\{\chi,\chi^{\sigma}\} \subseteq \{\psi \in \Irr(G)|\Q(\psi)=\Q(\chi)\}$. Suppose that there exists $\psi \in \Irr(G)\setminus \{\chi,\chi^{\sigma}\}$ with $\Q(\psi)=\Q(\chi)$. Then $\chi,\chi^{\sigma},\psi,\psi^{\sigma}$ are four irreducible characters with the same field of values, which contradicts that $f(G)\leq 3$.
\end{proof}
\end{lem}

As a consequence, if $f(G)\leq 3$, we  deduce that for each  quadratic extension $F$ of $\Q$, there exist at most two irreducible characters of $G$ whose field of values is $F$.

Let $n$ be a positive integer, we define the cyclotomic extension of order $n$, as $\Q_{n}=\Q(e^{\frac{2i\pi }{n}})$. We recall that for every $\chi \in \Irr(G)$ and for every $g\in G$, $\Q(\chi(g))\in \Q_{o(g)}$. The following two lemmas will be useful to deal with $\Q_{o(g)}$, where $g \in G$.

\begin{lem}\label{order}
Assume that $G/G''=\mathsf{F}_{rq}$, where $q$ is a prime $G/G'\cong \mathsf{C}_{r}$ is the Frobenius complement of $\mathsf{F}_{rq}$ and that $G''$ is a $p$-elementary abelian group. Then $o(g)$ divides $rp$, for  every $g \in G\setminus G'$. 
\end{lem}

\begin{lem}
Let $n$ be a positive integer. Then the following hold.
\begin{itemize}
\item[(i)] If $n=p$, where $p$ is an odd prime, then  $\Q_{n}$ contains only one quadratic extension.

\item[(ii)] If $n=p$, where $p$ is an odd prime, then  $\Q_{n}$ contains only one cubic extension  if $n\equiv 1 \pmod 3$ and contains no cubic extension if $n\not \equiv 1 \pmod 3$.

\item[(iii)] If $n=p^{k}$, where $p$  is an odd prime and $k\geq 2$, then  $\Q_{n}$ contains only one quadratic extension.

\item[(iv)] If $n=p^{k}$, where $p$  is an odd prime and $k\geq 2$, then  $\Q_{n}$ contains  one cubic  extension   if  $p\equiv 1 \pmod 3$ or $p=3$ and contains no cubic extension  if $p\equiv -1 \pmod 3$.

\item[(v)] If $n=p^{k}q^{t}$, where $p$ and $q$ are odd primes and $k,t \geq 1$, then  $\Q_{n}$ contains $3$ quadratic extensions.

\item[(vi)] If $n=p^{k}q^{t}$, where $p$ and $q$ are odd primes and $k,t \geq 1$, then  $\Q_{n}$  contains $4$ cubic extensions if both $\Q_{p^k}$ and $\Q_{q^t}$ contain  cubic extensions, contains one cubic extensions if only one of $\Q_{p^k}$ or $\Q_{q^t}$ contains a  cubic extension and does not contain cubic extensions if both $\Q_{p^k}$ and $\Q_{q^t}$ do not  contain cubic extensions.

\item[(vii)] If $n$ is odd, then $\Q_{n}=\Q_{2n}$.
\end{itemize}
\begin{proof}
This result follows from elementary Galois Theory. As an example, we prove (iii) and (iv). We know that $\Gal(\Q_{p^k}/\Q)\cong \mathsf{C}_{p^{k-1}(p-1)}$. Since $\Q_{p^k}$ has as many quadratic extensions as the number subgroups of index $2$ in $\Gal(\Q_{p^k}/\Q)$, we deduce that $\Q_{p^k}$ has only one quadratic extension. Now, we observe that $\Q_{p^k}$ has cubic extensions if and only if $3$ divides $p^{k-1}(p-1)$. This occurs if and only if $p=3$ or if $3$ divides $p-1$. If $\Q_{p^k}$ has cubic extensions, we can argue as in the quadratic case to prove that it has only one cubic extension. Thus, (iv) follows.
\end{proof}
\end{lem}

The following is well known.

\begin{lem}\label{exten}
Let $N$ be a normal subgroup of $G$ and let $\theta \in \Irr(N)$ be invariant in $G$. If $(|G:N|,o(\theta)\theta(1))=1$, then there exists a unique $\chi \in \Irr(G)$ such that $\chi_{N}=\theta$, $o(\chi)=o(\theta)$ and $\Q(\chi)=\Q(\theta)$. In particular, if $(|G:N|,|N|)=1$, then every invariant character of $N$ has an unique extension to $G$ with the same order and the same field of values.
\begin{proof}
By Theorem 6.28 of \cite{Isaacscar}, there exists $\chi$ an unique extension such that $o(\chi)=o(\theta)$. Clearly, $\Q(\theta) \subseteq \Q(\chi)$. Assume that $\Q(\theta) \not=\Q(\chi)$, then there exists $\sigma \in \Gal(\Q(\chi)/\Q(\theta))\setminus\{1\}$. Then $\chi^{\sigma}$ extends $\theta$ and $o(\chi)=o(\theta)=o(\chi^{\sigma})$, by unicity of $\chi$ that is impossible. Thus,   $\Q(\theta) =\Q(\chi)$ as we claimed.
\end{proof}
\end{lem}

We need to introduce some notation in order to state the results deduced from \cite{VeraLopez}. If $G$ is a finite group, then we write $k(G)$ to denote the number of conjugacy classes of $G$ and $\alpha(G)$ to denote the number of $G$-conjugacy classes contained in $G\setminus S(G)$, where $S(G)$ is the socle of $G$.

\begin{thm}\label{Vera-Lopez}
Let $G$ be a group such that $k(G)\leq 11$. If $f(G)\leq 3$, then $G \in \{\mathsf{C}_{2},\mathsf{C}_{3},\mathsf{C}_{4},\mathsf{D}_{10},\mathsf{A}_{4},\mathsf{F}_{21},\mathsf{S}_{3},\mathsf{D}_{14},\mathsf{D}_{18},\mathsf{F}_{20},\mathsf{F}_{52},\mathsf{A}_{5}, \PSL(2,8),\Sz(8)\}$.
\begin{proof}
Using the classification of  \cite{VeraLopez} of groups with $k(G)\leq 11$, we can see that these are the only groups with $f(G)\leq 3$ and $k(G)\leq 11$.
\end{proof}
\end{thm}

\begin{thm}\label{Vera-Lopez3}
Let $G$ be a solvable group with $\alpha(G)\leq 3$. Then either $G=\mathsf{S}_4$ or $G$ is metabelian.
\begin{proof}
If $G$ is a  group with  $\alpha(G) \leq 3$, then $G$ must be one of the examples listed in Lemmas 2.18, 2.19 and 2.20 of \cite{VeraLopez}. We see that except for $\mathsf{S}_4$ all solvable groups in those lemmas are metabelian.
\end{proof}
\end{thm}

\begin{thm}\label{Vera-Lopez2}
Let $G$ be a  group such that $S(G)$ is abelian, $k(G)\geq 12$, $4 \leq \alpha(G) \leq 9$ and $k(G/S(G))\leq 10$. Then $f(G)>3$.
\begin{proof}
If $G$ is a  group such that  $4 \leq \alpha(G) \leq 10$ and $k(G/S(G))\leq 10$, then $G$ must be one of the examples listed in Lemmas 4.2, 4.5, 4.8, 4.11, 4.14 of \cite{VeraLopez}. We see that $f(G)>3$ for all groups in those lemmas with $k(G)>11$.
\end{proof}
\end{thm}

Now, we classify all nilpotent groups with $f(G)\leq 3 $. 

\begin{thm}\label{nilpotent}
If $G$ is a nilpotent group with $f(G)\leq 3,$ then $G \in \{\mathsf{C}_{2},\mathsf{C}_{3},\mathsf{C}_{4}\}$.
\begin{proof}
 Let $p$ be a prime dividing $|G|$. Then there exists $K\trianglelefteq G$  such that $G/K=\mathsf{C}_{p}$.  Therefore, $f(\mathsf{C}_{p})= f(G/K)\leq f(G)\leq3$, and hence $p \in \{2,3\}$. Thus, the set of prime divisors of $|G|$ is contained in $\{2,3\}$.

If $6$ divides $|G|$, then there exists  $N$, a normal subgroup of $G$, such that $G/N=\mathsf{C}_{6}$. However, $f(\mathsf{C}_{6})=4> 3$ and we deduce that $G$ must be a $p$-group. It follows that $G/\Phi(G)$ is an elementary abelian $2$-group or an elementary abelian $3$-group with $f(G/\Phi(G)) \leq 3$. Since $f(\mathsf{C}_{2}\times \mathsf{C}_{2})=4$ and $f(\mathsf{C}_{3}\times \mathsf{C}_{3})=8$, we have that $G/\Phi(G) \in \{\mathsf{C}_{2},\mathsf{C}_{3}\}$. Thus, $G$ is a cyclic $2$-group or a  cyclic $3$-group. Since $f(\mathsf{C}_{8})>3$ and $f(\mathsf{C}_{9})>3$, it follows that $G\in \{\mathsf{C}_{2},\mathsf{C}_{4},\mathsf{C}_{3}\}$. 
\end{proof}
\end{thm}

In the remaining we will assume that $G$ is not a nilpotent group. From this case, we can also deduce the following result.

\begin{cor}\label{der}
If $G$ is group with $f(G)\leq3$, then either $G=G'$ or  $G/G' \in \{\mathsf{C}_{2},\mathsf{C}_{3},\mathsf{C}_{4}\}$.
\begin{proof}
Suppose that $G'<G$, then $G/G'$ is an abelian group with $f(G/G')\leq 3$. Thus, by Theorem \ref{nilpotent}, $G/G' \in \{\mathsf{C}_{2},\mathsf{C}_{3},\mathsf{C}_{4}\}$.
\end{proof}
\end{cor}

In the proof of the solvable case of Theorem A, we need to see that there are no groups $G$ with $f(G)\leq 3$ of certain orders. We collect them in the next result.

\begin{lem}\label{casos}
There exists no  group $G$ with $f(G)\leq 3$ and $|G| \in \{30,42, 48,50,54,\\70,84,98,100,126,147,156,234,260,342,558,666,676,774,882,903,954,1098,1206,\\1314,1404,2756,4108,6812,8164\}$.
\begin{proof}
We observe that all numbers in the above list are smaller than 2000, except $\{2756,4108,6812,8164\}$. However, the numbers $\{2756,4108,6812,8164\}$ are cube-free. Thus, we can use GAP \cite{gap} to check the result.
\end{proof}
\end{lem}

\section{Non-solvable case}\label{Section3}
In this section we  classify the  non-solvable groups with $f(G)\leq 3$.

\begin{thm}\label{nonsolvable}
Let $G$ be a non-solvable group with $f(G)\leq 3$. Then $f(G)\leq 3$ and $G \in \{\mathsf{A}_{5}, \PSL(2,8), \Sz(8)\}$.
\end{thm}

If $G$ is a group with $f(G)\leq 3$, it follows trivially that $G$ possesses at most $3$ irreducible rational characters. We will use the following results from \cite{Navarro-Tiep} and \cite{Rossi}, which classify the non-solvable groups with two or three rational characters, respectively.

\begin{thm}[Theorems B and C of \cite{Navarro-Tiep}]\label{Navarro-Tiep}
Let $G$ be a non-solvable group. Then $G$ has at least 2 irreducible rational characters. If moreover, $G$ has exactly two irreducible rational characters, then $M/N \cong \PSL(2,3^{2a+1})$, where $M=O^{2'}(G)$, $N=O_{2'}(M)$ and $a \geq 1$.
\end{thm}

\begin{thm}[Theorem B of \cite{Rossi}]\label{simplePrev2}
Let $G$ be a non-solvable group with exactly three rational characters. If $M:=O^{2'}(G)$, then there exists $N\triangleleft G$ solvable and contained in $M$ such that $M/N$ is one of the following groups:
\begin{itemize}
\item[(i)] $\PSL(2,2^{n})$, where $n\geq2$.

\item[(ii)] $\PSL(2,q)$, where $q\equiv 5 \pmod{24}$ or $q\equiv-5 \pmod{24}$.

\item[(iii)] $\Sz(2^{2t+1})$, where $t \geq 1$.

\item[(iv)] $ \PSL(2,3^{2a+1})$, where $a \geq 1$.
\end{itemize}
If moreover $M/N$ has the form (i),(ii) or (iii), then $N=O_{2'}(M)$.
\end{thm}

From Theorems  \ref{Navarro-Tiep} and \ref{simplePrev2}, we deduce that  if $S$ is a simple group with  at most three rational characters, then $S$ is one of the groups listed above. That will allow us to determine the simple groups with $f(G)\leq 3$.   Looking at  the character tables of the groups $\PSL(2,q)$ (see \cite{Dornhoff}, chapter 38) and $\Sz(q)$ (see \cite{Geck}), we see that there is always an entry of the form $e^{\frac{-2\pi i}{q-1}}+e^{\frac{-2\pi i}{q-1}}$. For this reason, we  study whether  $e^{\frac{2\pi i}{r}}+e^{\frac{2\pi i}{r}}$ is rational, quadratic or cubic. Let $r$ be a positive integer. We will write $\varphi(r)$ to denote the Euler's function of $r$, this is $\varphi(r)=|\{k\in \{1,\ldots,r-1\}| (k,r)=1\}|$.

\begin{lem}\label{omega}
Let $r$ be a positive integer, let $\nu=e^{\frac{2\pi i}{r}}$ and let $\omega=\nu+\nu^{-1}$. Then the following hold

\begin{itemize}
\item[(i)] $\omega$ is rational if and only if $r\in \{3,4,6\}$.

\item[(ii)] $\omega$ is quadratic if and only if $r\in \{5,8,10\}$.

\item[(iii)] $\omega$ is cubic  if and only if $r\in \{7,9,14,18\}$.
\end{itemize}
\begin{proof}
Let $k\in \{1,\ldots,r-1\}$ such that $(r,k)=1$. Then there exists $\sigma_{k} \in \Gal(\Q(\nu)/\Q)$ such that $\sigma_{k}(\nu)=\nu^{k}$ and hence $\sigma_{k}(\omega)=\nu^{k}+\nu^{-k}$.

Suppose that $\omega\in \Q$. Let $k\in \{2,\ldots,r-1\}$ such that $(r,k)=1$.  Since $\sigma_{k}(\omega)=\omega$, we deduce that $k=r-1$. Thus,  we deduce that $\varphi(r)=2$ and hence $r\in \{3,4,6\}$.

Suppose now that $\omega$ is quadratic. Then there exists  $\sigma \in \Gal(\Q(\nu)/\Q)$ such that $\sigma(\omega)\not=\omega$. We deduce that $\sigma(\nu)=\nu^{k_{0}}$, where $k_{0} \in \{2,\ldots,r-2\}$ and $(r,k_{0})=1$. Since $\omega$ is quadratic, it follows that $\sigma(\omega)$ is the only Galois conjugate of $\omega$ and hence $\{k \leq r|(r,k)=1\}=\{1,k_{0},r-k_{0},r-1\}$. Thus, $\varphi(r)=4$ and  (ii) follows.

Reasoning as in the previous case, we can deduce that  $\omega$ is cubic if and only if  $\varphi(r)= 6$ and hence (iii) follows.
\end{proof}
\end{lem}

\begin{thm}\label{simple}
Let $S$ be a non-abelian simple group with $f(S)\leq 3$. Then $S \in \{\mathsf{A}_{5},\PSL(2,8),\Sz(8)\}$.
\begin{proof} 
Since $f(S)\leq 3$,  $S$ has at most three rational characters. Thus, $S$ has the form described in Theorem \ref{simplePrev2}. We claim that the only groups in those families with $f(S)\leq3$ are $\mathsf{A}_{5}(=\PSL(2,4))$, $\PSL(2,8)$ and $\Sz(8)$.

Let $S=\PSL(2,q)$ where $q$ is a prime power or let $S=\Sz(q)$ where $q=2^{2t+1}$ and $t\geq 1$. We know that there exists $\chi \in \Irr(S)$ and $a \in S$ such that $\chi(a)=e^{\frac{-2\pi i}{q-1}}+e^{\frac{-2\pi i}{q-1}}$. The condition $f(S)\leq 3$ implies that $|\Q(\chi(a)):\Q|\leq 3$. By Lemma \ref{omega}, we deduce that $q-1 \in \{3,4,5,6,7,8,9,10,14,18\}$.  If $S=\PSL(2,q)$, we have that  $q=2^n$, $q=3^{2m+1}$ or $q\equiv \pm 5 \pmod{24}$. Thus, we only have to consider the cases $q \in \{5,8,19\}$. Finally, we have that $3=f(\PSL(2,5))=f(\PSL(2,8))$ and $f(\PSL(2,19))=4$. If $S=\Sz(q)$, we have that $q=2^{2t+1}$ and hence we only have to consider the case $q=8$. Finally, we have that $f(\Sz(8))=3$.

Thus, the only simple groups  with $f(S)=3$ are $\mathsf{A}_{5}$, $\PSL(2,8)$ and $\Sz(8)$.
\end{proof}
\end{thm}

Using Theorem \ref{Navarro-Tiep} we  prove  that a non-solvable group with $f(G)\leq 3$ has exactly three rational characters.

\begin{thm}\label{2racional}
Let $G$ be a non-solvable group with $f(G)\leq 3$. Then $G$ has exactly three rational irreducible characters. In particular, $f(G)=3$.
\begin{proof}
By Theorem \ref{Navarro-Tiep}, $G$ has at least two rational irreducible characters. Suppose that $G$ has exactly  two rational irreducible characters. Applying again Theorem  \ref{Navarro-Tiep}, if $M=O^{2'}(G)$ and $N=O_{2'}(M)$, then $M/N \cong \PSL(2,3^{2a+1})$. Taking the quotient by $N$, we may assume that $N=1$.

By Theorem  \ref{simple}, $f(M)=f(\PSL(2,3^{2a+1}))>3$ and hence  we deduce that $M<G$.  Now, we claim that there exists a rational character of $M$ that can be extended to a rational character of $G$. By Lemma 4.1 of \cite{Auto}, there exists $\psi \in \Irr(M)$, which is rational and is extendible to a rational character $\varphi \in \Irr(\Aut(M))$. If $H=G/\mathsf{C}_{G}(M)$, then we can identify $H$ with a subgroup of $\Aut(M)$ which contains $M$. Therefore, $\varphi_{H}\in \Irr(H)\subseteq \Irr(G)$ and it is rational, as we wanted.

Let $\chi \in \Irr(G/M)\setminus\{1\}$. Since $|G/M|$ is odd,  $\chi$ cannot be rational. Thus, there exists $\rho\not =\chi$, a Galois conjugate of $\chi$. Then $\Q(\chi)=\Q(\rho)$. Since $\psi$ is extendible to the rational character $\varphi \in \Irr(G)$, applying Gallagher's Theorem  (See Corollary 6.17 of \cite{Isaacscar}), we have that $\chi \varphi\not=\rho \varphi$ are two irreducible characters of $G$ and $\Q(\chi)=\Q(\rho)=\Q(\varphi\chi)=\Q(\varphi\rho)$. Therefore, we have $4$  irreducible characters with the same field of values, which is impossible.
\end{proof}
\end{thm}

Now, we  use Theorem \ref{simplePrev2} to determine $G/O_{2'}(G)$.

\begin{thm}\label{reduction}
Let $G$ be a finite non-solvable  group with $f(G)=3$. Then $G/O_{2'}(G) \in \{\mathsf{A}_{5},\PSL(2,8),\Sz(8)\}$.

\begin{proof}
Let $M$ and $N$ be as in Theorem \ref{simplePrev2}. We assume for the moment that $N=1$.

Suppose first that $M<G$. Reasoning as in Theorem \ref{2racional}, we can prove that there exists $\psi \in \Irr(M)$ such that it is extendible to a rational character $\varphi \in \Irr(G)$. As in Theorem \ref{2racional}, if we take $\chi \in \Irr(G/M)\setminus\{1\}$ and $\rho$ a Galois conjugate of $\chi$, then  $\Q(\chi)=\Q(\rho)=\Q(\varphi\chi)=\Q(\varphi\rho)$, where all of these characters are different, which is a contradiction.

Thus, $M=G$ and hence $G$ is a simple group with $f(G)=3$. By Theorem \ref{simple}, $G\in \{\mathsf{A}_{5},\PSL(2,8),\Sz(8)\}$.

If we apply the previous reasoning to $G/N$, then  we have that $G/N$ is one of the desired groups. In either case, $G/N$ has the form  (i),(ii) or (iii) of  Theorem \ref{simplePrev2} and hence $N=O_{2'}(G)$.
\end{proof}
\end{thm}

To complete our proof it only remains to prove that $O_{2'}(G)=1$. However, we need to study before two special cases.  First, we  study the case when $O_{2'}(G)=Z(G)$.

\begin{thm}\label{quasisimple}
There is no quasisimple group $G$ such that $O_{2'}(G)=Z(G)$, $O_{2'}(G)>1$ and $G/Z(G) \in \{\mathsf{A}_{5},\PSL(2,8),\Sz(8)\}$.
\begin{proof}
Suppose that such a group exists. Then we have that $|Z(G)|$ divides $|M(S)|$, where $S=G/Z(G)$. Since the Schur multiplier of $\mathsf{A}_{5}$, $\Sz(8)$ and $\PSL(2,8)$ is $\mathsf{C}_{2}$, $\mathsf{C}_{2}\times \mathsf{C}_{2}$ and the trivial group, respectively, we have that $Z(G)$ is a $2$-group. However, $Z(G)=O_{2'}(G)$ and hence $|Z(G)|$ has odd order. Thus, $Z(G)=1$ and the result follows.
\end{proof}
\end{thm}

We need to introduce more notation to deal with the remaining case.  For any group $G$, we define $o(G)=\{o(g)|g \in G \setminus \{1\}\}$. Suppose that $f(G)\leq 3$  and let $\chi \in \Irr(G)$ be a non-rational character. Then $\Q(\chi)=\Q(\chi(g))$ for some $g \in G \setminus \{1\}$. Thus, $\Q(\chi)$ is a quadratic extension or a cubic extension of $\Q_{n}$, where $n = o(g)$. If $N$ is a normal subgroup of $G$, then we write $\Irr(G|N)$ to denote the set of $\chi \in \Irr(G)$ such that $N \not \leq \ker(\chi)$. Finally,  if $N$ is a normal subgroup of $G$ and $\theta \in \Irr(N)$, then we write $I_{G}(\theta)=\{g \in G|\theta^{g}=\theta\}$ to denote the inertia subgroup of $\theta$ in $G$.

\begin{thm}\label{other}
There is no  group $G$ with $f(G)\leq 3$ such that $G/O_{2'}(G) \in \{\mathsf{A}_{5},\\ \PSL(2,8), \Sz(8)\}$, $O_{2'}(G)$ is elementary abelian and a $G/O_{2'}(G)$-simple module.
\begin{proof}
Write $V=O_{2'}(G)$  and let $|V|=p^d$ with $p>2$. Thus, if $\F_{p}$ is the field of $p$ elements, then $V$ can be viewed as an irreducible $\F_{p}[G/V]$-module of dimension $d$. We can extend the associated representation  to a representation  of $G/V$ over an algebraically closed field in characteristic $p$. Thus, the representation given by $V$ can be expressed as a sum of irreducible representations of $G/V$ over an algebraically closed field in characteristic $p$. Let $m(S)$ be the smallest degree of a non-linear $p$-Brauer character of $S$.  We have that $d \geq m(G/V)$.  We have to distinguish two different cases:  $p$ divides $|G/V|$ and $p$ does not divide $|G/V|$.

\underline{Case $p$ does not divide $|G/V|$:} In this case the Brauer characters are  the ordinary characters. Thus,  $|V|=p^{d}$ where $d$ is at least the smallest degree of an irreducible non-trivial character of $G/V$. Now, let $\lambda \in \Irr(V)\setminus \{1\}$. Then $\Q(\lambda)\subseteq \Q_{p}$. Since $(|G/V|,|V|)=1$,  we have that $(|I_{G}(\lambda)/V|,|V|)=1$. Thus, by Lemma \ref{exten}, we have that $\lambda$ has an extension $\psi \in \Irr(I_{G}(\lambda))$ with $\Q(\psi)=\Q(\lambda)\subseteq \Q_{p}$. By the Clifford's correspondence  (See Theorem 6.11 of \cite{Isaacscar}) $\psi^{G} \in \Irr(G)$ and $\Q(\psi^{G})\subseteq \Q(\psi) \subseteq \Q_{p}$. Thus, given $\zeta$,  an  orbit of $G/V$ on $\Irr(V)\setminus \{1_{V}\}$, there exists $\chi_{\zeta} \in \Irr(G|V)$ such that $\Q(\chi_{\zeta})\subseteq \Q_{p}$. Let $F$ be the unique quadratic extension of $\Q_{p}$ and let $T$ be the unique cubic extension of $\Q_{p}$ (if such an extension exists). Since $\Irr(G/V)$ contains three rational characters, we deduce that $\Q(\chi_{\zeta})\in \{T,F\}$ and since $F$ is quadratic, then there are at most $2$ characters whose field of values is $F$. Thus, the action of $G/V$ on $\Irr(V)\setminus \{1_{V}\}$ has at most $5$ orbits. Therefore, $|V|=|\Irr(V)|\leq 5|G/V|+1$.

\begin{itemize}
\item[(i)] Case $G/V=\mathsf{A}_{5}$: In this case $|V|\geq 7^3=343$ (because $7$ is the smallest prime not dividing $|G/V|$ and $3$ is the smallest degree of a non-linear character of $\mathsf{A}_{5}$).  On the other hand, we have that  $|V|\leq 5|G/V|+1\leq 5\cdot 60+1=301<343$, which is a contradiction.

\item[(ii)] Case $G/V=\PSL(2,8)$:  In this case  $|V|\geq 5^{7}=78125$ and $|V|\leq 5\cdot504+1=2521$, which is a contradiction.

\item[(iii)]  Case $G/V=\Sz(8)$: In this case $|V|\geq 3^{14}=4782969$ and $|V|\leq 5\cdot 29120+1=145601$, which is a contradiction.
\end{itemize}

\underline{Case $p$ divides $G/V$:} From the Brauer character tables of $\{\mathsf{A}_{5},\PSL(2,8),\\ \Sz(8)\}$, we deduce that  $m(\mathsf{A}_{5})=3$ for $p \in \{3,5\}$,  $m(\PSL(2,8))=7$ for $p \in \{3,7\}$ and $m(\Sz(8))=14$ for $p \in \{3,7,13\}$.

\begin{itemize}
    \item [(i)] Case  $G/V=\PSL(2,8)$:
    \begin{itemize}
        \item [a)] $p=7$: In this case  $|V|=7^{d}$ with $d\geq 7$ and $o(G)=\{2,3,7,9,2\cdot 7, 3\cdot 7, 7 \cdot 7, 9 \cdot 7\}$. On the one hand,  the number of non-trivial  $G$-conjugacy classes contained in $V$ is at least $\frac{|V|}{|G/V|}\geq \frac{7^{7}}{504}\geq 1634$. Therefore, we deduce that $|\Irr(G)|\geq 1634$. On the other hand,   we have that there are at most $3$ quadratic extensions and at most $4$ cubic extensions contained in $\Q_{n}$, where $n \in o(G)$. Applying again that $f(G)\leq 3$, we have that the number of non-rational characters in $G$ is at most $2\cdot3+3\cdot 4=18$. Counting the rational characters, we have that $|\Irr(G)|\leq 21<1634$, which is a contradiction.

\item [b)] $p=3$: In this case $|V|=3^{d}$ with $d\geq 7$ and by calculation $k(G)=|\Irr(G)|\leq 3+2\cdot 3+3\cdot 2=15$.  We know that $V=S(G)$, and hence if $4\leq \alpha(G)\leq 9$, then $f(G)>3$ by Theorem \ref{Vera-Lopez2} (clearly $\alpha(G)\geq 4$ because $k(G/S(G))=9$). Thus,  $\alpha(G)\geq 10$. Since $V=S(G)$ and $k(G)\leq 15$, we deduce that $V$ contains at most $4$ non-trivial $G$-conjugacy classes. Thus, $|V|\leq 504\cdot 4+1=2017<3^{7}$ and hence we have a contradiction.
    \end{itemize}

\item [(ii)] Case $G/V=\Sz(8)$: In this case $|V|\geq 5^{14}$ and as before $|\Irr(G)|\geq 209598$.
    \begin{itemize}
        \item [a)] $p=5$: By calculation, $|\Irr(G)|\leq 3 +2 \cdot 7+3\cdot 2=23<209598$, which is a contradiction.

\item [b)] $p\in \{7,13\}$: By calculation, $|\Irr(G)|\leq 3+2\cdot 7+3\cdot 4 =29<209598$, which is a contradiction.
    \end{itemize}

\item [(iii)] Case $G/V=\mathsf{A}_{5}$:
   \begin{itemize}
        \item [a)] $p=3$: In this case $|V|=3^d$, where $d\geq 3$ and by calculation we have that, $|\Irr(G)|\leq 3+ 2\cdot 3+3 \cdot 1 =12$. As before, applying Theorem \ref{Vera-Lopez2}, we can deduce that $|V|$ contains at most one non-trivial $G$-conjugacy class. Thus, $|V|\leq 61$ and since $V$ is a 3-group we deduce that $|V|= 3^3$.  We also deduce that $26$ is the size of a $G$-conjugacy class.  That is impossible since 26 does not divide $|G/V|=60$. 

\item [b)] $p=5$: In this case $k(G)\leq 9$ and by Theorem \ref{Vera-Lopez} there is no group with the required properties.
    \end{itemize}

\end{itemize}

We conclude that there is no group with the desired form and hence $V=1$.
\end{proof}
\end{thm}

Now, we are prepared to prove of Theorem \ref{nonsolvable}

\begin{proof}[Proof of Theorem \ref{nonsolvable}]
By Theorem \ref{reduction}, we know that $G/O_{2'}(G) \in \{\mathsf{A}_{5},\PSL(2,8), \\\Sz(8)\}$. We want to prove that $O_{2'}(G)=1$. Suppose that $O_{2'}(G)>1$. Taking an appropriate quotient, we may assume that $O_{2'}(G)$ is a minimal normal subgroup of $G$. Since $O_{2'}(G)$ is solvable, we have that $O_{2'}(G)$ is a $p$-elementary abelian subgroup for some odd prime $p$. There are two possibilities for $O_{2'}(G)$. The first one is that   $O_{2'}=Z(G)$, and the second one is that $O_{2'}(G)$ is a $G/O_{2'}(G)$-simple module. The first one is impossible by Theorem \ref{quasisimple} and the second one is impossible by Theorem \ref{other}. Thus, $O_{2'}(G)=1$ and the result follows.
\end{proof}

Therefore, the only non-solvable groups with $f(G)\leq 3$ are $\mathsf{A}_{5},\PSL(2,8)$ and $\Sz(8)$. In the remaining we will assume that $G$ is solvable.

\section{Metabelian case}\label{Section5}

Let $G$ be a finite metabelian group with $f(G)\leq 3$. By Corollary \ref{der}, we have that $G/G' \in \{\mathsf{C}_{2},\mathsf{C}_{3},\mathsf{C}_{4}\}$ and hence we can divide  this case  in different subcases. We begin by studying the case when $G'$ is $p$-elementary abelian.

\begin{lem}\label{casopelem}
Let $G$ be a  finite group such that $f(G)\leq 3$ and $G'\not=1$ is $p$-elementary abelian. Then $G \in \{\mathsf{S}_{3},\mathsf{D}_{10},\mathsf{A}_{4},\mathsf{D}_{14},\mathsf{F}_{21},\mathsf{F}_{20},\mathsf{F}_{52}\}$.
\begin{proof}
First, we observe that $(|G:G'|,p)=1$. Otherwise, $G$ would be a nilpotent group with $f(G)\leq 3$. Thus, by Theorem \ref{nilpotent}, we would have  that $G'=1$, which is impossible.

Let $\psi \in \Irr(G')\setminus \{1_{G'}\}$ and let $I_{G}(\psi)$ be the inertia group of $\psi$ in $G$. Since $G/G'$ is cyclic, applying Theorem 11.22 of \cite{Isaacscar}, we have that $\psi$ can be extended to an irreducible character of $I_{G}(\psi)$. Since $\psi$ cannot be extended to $G$,  we have that $\psi$ cannot be invariant and hence $I_{G}(\psi)<G$. Now, we will study separately the case $G/G' \in \{\mathsf{C}_{2},\mathsf{C}_{3}\}$ and the case $G/G'=\mathsf{C}_{4}$.

Assume first that $G/G' \in \{\mathsf{C}_{2},\mathsf{C}_{3}\}$. Since $ I_{G}(\psi)< G$, we deuce that $I_{G}(\psi)=G'$ for every $\psi \in \Irr(G')\setminus \{1_{G'}\}$. Thus,  by Clifford correspondence, $\psi^G\in \Irr(G)$.

Therefore, if $\chi \in \Irr(G|G')$, then $\chi$  has the form $\chi=\psi^{G}$, where $\psi \in \Irr(G')\setminus \{1_{G'}\}$.  Since  $\mathbb{Q}(\psi)\subseteq \mathbb{Q}_{p}$, we have that $\mathbb{Q}(\psi^{G})\subseteq \mathbb{Q}_{p}$. We know that there exists at most one quadratic extension in $\mathbb{Q}_{p}$ and at most one cubic extension in $\mathbb{Q}_{p}$.  Since $\Irr(G/G')$ contains at least one rational character and $f(G)\leq 3$, we have that $|\Irr(G|G')|\leq 2+1\cdot 2+ 1\cdot 3=7$. Since $|\Irr(G/G')|\leq 3$, we have that $k(G)=|\Irr(G)| = |\Irr(G|G')|+|\Irr(G/G')|\leq 7+3=10$.  By Theorem \ref{Vera-Lopez}, we deduce that the only groups such that $|G:G'|\in \{2,3\}$, $G'$ is elementary abelian, $f(G)\leq 3$ and $k(G)\leq 10$ are $\{\mathsf{S}_{3},\mathsf{D}_{10},\mathsf{A}_{4},\mathsf{D}_{14},\mathsf{F}_{21}\}$.

Assume now that $G/G'=\mathsf{C}_{4}$. If $\psi \in \Irr(G')\setminus \{1_{G'}\}$, we have that $I_{G}(\psi)<G$ and hence we have two possible options.

 The first one is that $I_{G}(\psi)=G'$. In this case, applying the Clifford correspondence, we have that $\psi^{G}\in \Irr(G)$ and hence $\mathbb{Q}(\psi^{G})\subseteq \Q(\psi)\subseteq \mathbb{Q}_{p}$.  The other one is that $|G:I_{G}(\psi)|=2$. In this case, applying Lemma \ref{exten}, we have that $\psi $ is extendible to $\varphi \in \Irr(I_{G}(\psi))$ and $\Q(\varphi)=\Q(\psi)\subseteq \Q_{p}$.  Let $\Irr(I_{G}(\psi)/G')=\{1,\rho\}$. By Gallagher's Theorem, $\varphi$ and $\varphi\rho$ are all the extensions of $\psi$ to $I_{G}(\psi)$. Since $\Q(\rho)=\Q$, we have that $\Q(\varphi\rho)=\Q(\varphi)\subseteq \Q_{p}$. Let $\tau \in \{\varphi,\varphi\rho\}$. We  have that $\tau^{G} \in \Irr(G)$, and hence $\Q(\tau^{G})\subseteq \Q(\tau)\subseteq \Q_{p}$. Therefore, $\Q(\chi)\subseteq \Q_{p}$ for every  $\chi \in \Irr(G|G')$.

As before, we can deduce that $ \Irr(G|G')$ contains at most $5$ non-rational characters. On the other hand, $\Irr(G/G')$ contains two rational characters and hence $\Irr(G|G')$ contains at most one rational character. Therefore, $|\Irr(G|G')|\leq 6$ and hence $k(G)=|\Irr(G/G')|+|\Irr(G|G')|\leq 4+6=10$. By Theorem \ref{Vera-Lopez}, our only possible options are  $\{\mathsf{F}_{20},\mathsf{F}_{52}\}$.
\end{proof}
\end{lem}

\begin{thm}\label{caso2ab}
Let $G$ be a  metabelian group with $f(G)\leq 3$ such that $|G:G'|=2$. Then $G \in \{\mathsf{S}_{3},\mathsf{D}_{10},\mathsf{D}_{14},\mathsf{D}_{18}\}$.
\begin{proof}
 Assume for the moment that $G'$ is a $p$-group. We note  that $F(G)=G'$. Therefore, $G'/\Phi(G)=F(G)/\Phi(G)$ is $p$-elementary abelian. Thus, by Lemma \ref{casopelem}, we have that $G/\Phi(G) \in \{\mathsf{S}_{3},\mathsf{D}_{10},\mathsf{D}_{14}\}$ and hence $G'/\Phi(G)$ is cyclic. Therefore, $G'$ is a cyclic $p$-group and and we have only three possibilities for $p$. We  analyse the cases $p=3$, $p=5$ and $p=7$ separately.

If $p=3$, then  $G'$ is a cyclic group of order $3^{l}$. If $l \geq 3$, then there exists $K$ characteristic in $G'$  of order $3^{l-3}$. Thus, $|G/K|=2\cdot3^{3}=54$ and $f(G/K)\leq 3$. However, by Lemma \ref{casos}, there is no group of order $54$ with $f(G)\leq 3$. Thus, $l\in\{1,2\}$. If $l=1$, then $G=\mathsf{S}_{3}$ and if $l=2$, then  $G=\mathsf{D}_{18}$.

 If $p \in \{5,7\}$, then $G'$ is a cyclic group of order $p^{l}$. If $l \geq 2$, then there exists $K$ characteristic in $G'$  of order $p^{l-2}$. Thus, $|G/K|=2\cdot p^{2}$ and $f(G/K)\leq 3$. For $p=5$, we have that $|G/K|=2\cdot 5^{2}=50$ and for $p=7$, we have  that $|G/K|=2\cdot 7^{2}=98$. However, by Lemma \ref{casos}, there is no group of order $50$ or $98$ with $f(G)\leq3$.

Therefore, if $G'$ is a $p$-group, then $G \in \{\mathsf{S}_{3},\mathsf{D}_{18},\mathsf{D}_{10},\mathsf{D}_{14}\}$. From here, we also deduce that  the prime divisors of $|G'|$ are contained in $\{3,5,7\}$. To complete the classification it only remains to prove that $|G'|$ cannot be divided by two different primes. Suppose that both $3$ and $5$ divide $|G'|$. Taking  a quotient by a Sylow $7$-subgroup of $G'$, we may assume that the only prime divisors of $|G'|$ are $3$ and $5$. By the case when $|G'|$ is a $p$-group, we deduce that the Sylow $3$-subgroups and Sylow $5$-subgroups of $G'$ are both cyclic. Thus, $f(G/\Phi(G))\leq 3$  and $G'/\Phi(G)=\mathsf{C}_{3}\times \mathsf{C}_{5}$. Therefore, $G/\Phi(G)$ is a group of order $30$ with $f(G/\Phi(G))\leq 3$, which is impossible by Lemma \ref{casos}. Analogously, we can prove that if any of the pairs $\{3,7\}$ or $\{5,7\}$ divides $|G'|$ at the same time, then there exists  a group $H$ with $f(H)\leq 3$ of order $42$ or $70$, respectively. Applying again Lemma \ref{casos}, we have a contradiction. Thus, $G'$ is a $p$-group and the result follows.
\end{proof}
\end{thm}

\begin{thm}\label{caso3ab}
Let $G$ be a  metabelian group with $f(G)\leq 3$ such that   $|G:G'|=3$. Then $G \in \{\mathsf{A}_{4},\mathsf{F}_{21}\}$.
\begin{proof}
As in Theorem \ref{caso2ab}, we  assume first that $G'$ is a $p$-group. By Proposition \ref{casopelem}, we have that $G/\Phi(G) \in \{\mathsf{A}_{4},\mathsf{F}_{21}\}$.  Therefore, we have that $p\in \{2,7\}$. We analyse each case separately.

If $p=7$, then $G'/\Phi(G)=\mathsf{C}_{7}$. Thus, $G'$ is a cyclic group of order $7^{l}$. If $l \geq 2$, then there exists $K$ characteristic in $G'$  of order $7^{l-2}$. Thus, $|G/K|=3\cdot7^{2}=147$ and $f(G/K)\leq 3$. However, by Lemma \ref{casos}, there is no group of order $147$ with $f(G)\leq 3$. Thus, $l=1$ and hence $G= \mathsf{F}_{21}$.

If $p=2$, then $G'/\Phi(G)=\mathsf{C}_{2}\times \mathsf{C}_{2}$. Thus, $G'=U\times V$, where $U$ is cyclic of order $2^n$, $V$  is cyclic of order $2^m$ and $n\geq m$ .Then, we can take $H$ the unique subgroup of $U$ of order $2^{m}$.  Thus, $K=H\times V$ is normal in $G$ and $(G/K)'$ is a cyclic 2-group. Thus, $f(G/K)\leq 3$, $|G/K:(G/K)'|=3$ and $(G/K)'$ is a cyclic $2$-group, which is not possible by Proposition \ref{casopelem}. It follows that $n=m$ and hence $G'$ is a product of $2$ cycles of length $n$. If $n \geq 2$, then there exists $T$ characteristic in $G'$ such that $G'/T=\mathsf{C}_{4}\times \mathsf{C}_{4}$. Thus, $f(G/T)\leq 3$  and $|G/T|=48$, which contradicts Lemma \ref{casos}. It follows that $n=1$ and hence $G=\mathsf{A}_{4}$.

Therefore, we have that the  prime divisors of $G'$ are contained in $\{2,7\}$ and if $G'$ is a $p$-group, then $G \in  \{\mathsf{A}_{4},\mathsf{F}_{21}\}$. Assume now that both $2$ and $7$  divide $|G'|$. Then $G'/\Phi(G)=\mathsf{C}_{2}\times \mathsf{C}_{2}\times \mathsf{C}_{7}$. Thus, $|G/\Phi(G)|=84$ and $f(G/\Phi(G))\leq 3$, which is impossible by Lemma \ref{casos}. Then $G'$ must be a $p$-group and the result follows.
\end{proof}
\end{thm}

\begin{thm}\label{caso4ab}
Let $G$ be a metabelian group with $f(G)\leq 3$ such that  $|G:G'|=4$. Then $G \in \{\mathsf{F}_{20},\mathsf{F}_{52}\}$.
\begin{proof}
As in Theorem \ref{caso2ab}, we assume first that $G'$ is a $p$-group. By Lemma \ref{casopelem}, we have that $G/\Phi(G) \in  \{\mathsf{F}_{20},\mathsf{F}_{52}\}$ and hence $G'$ is a cyclic $p$-group, where $p \in \{5,13\}$.

 In both cases $G'$ is a cyclic group of order $p^{l}$. If $l \geq 2$, then there exists $K$ characteristic in $G'$ of order $p^{l-2}$. Thus, $|G/K|=4\cdot p^{2}$ and $f(G/K)\leq 3$. For $p=5$, we have that $|G/K|=4\cdot 5^{2}=100$ and for $p=13$, we have  that $|G/K|=4\cdot 13^{2}=676$. However, by Lemma \ref{casos} there is no group of order $100$ or $676$ with $f(G)\leq3$. 

Therefore, we have that the  prime divisors of $G'$ are contained in $\{5,13\}$ and if $G'$ is a $p$-group then $G \in \{\mathsf{F}_{20},\mathsf{F}_{52}\}$. Assume now that both $5$ and $13$ and divide $|G'|$. Then $G'/\Phi(G)= \mathsf{C}_{5}\times \mathsf{C}_{13}$. Thus, $f(G/\Phi(G))\leq 3$, $|G/\Phi(G)|=4\cdot 5 \cdot 13=260$, which contradicts Lemma \ref{casos}. Therefore, $G'$ must be a $p$-group and the result follows.
\end{proof}
\end{thm}

\section{Solvable case}

In this section we classify all solvable groups with $f(G)\leq 3$. By the results of the previous section, we have that $G/G'' \in \{\mathsf{C}_{2},\mathsf{C}_{3},\mathsf{C}_{4},\mathsf{S}_{3}, \mathsf{D}_{10},\mathsf{A}_{4},\mathsf{D}_{14},\mathsf{D}_{18},\mathsf{F}_{20},\mathsf{F}_{21},\mathsf{F}_{52}\}$. Therefore, the result will be completed once we prove that  $G''=1$. We will begin by  determining all possible $\Q(\chi)$ for $\chi \in \Irr(G|G'')$ and then, we will use this to bound $k(G)$. Finally, the result will follow from Theorems \ref{Vera-Lopez} and \ref{Vera-Lopez2} and some calculations.

\begin{lem}\label{restocasos}
Let $G$ be a group such that $G''\not=1$, $G''$ is $p$-elementary abelian, $G/G''\in \{\mathsf{S}_{3},\mathsf{D}_{10},\mathsf{D}_{14},\mathsf{F}_{21},\mathsf{F}_{20},\mathsf{F}_{52}\}$ and $p$ does not divide $|G'/G''|$.  If $r=|G:G'|$, then $\Q(\chi)\subseteq \Q_{rp}$ for every $\chi \in \Irr(G|G'')$.
\begin{proof}
By Lemma \ref{order}, we know that for every $g \in G \setminus G'$ and for every $\chi \in \Irr(G)$, $\chi(g) \in \Q_{rp}$. Therefore, we only have to prove that  $\Q(\chi_{G'})\subseteq \Q_{rp}$  for every $\chi \in \Irr(G|G'')$. It suffices to prove that $\Q(\psi)\subseteq \Q_{rp}$  for every $\psi \in \Irr(G'|G'')$.

Let $\lambda \in \Irr(G'')\setminus \{1_{G''}\}$. We know that $\Q(\lambda)\subseteq \Q_{p}$ and  $\lambda$ cannot be extended to an irreducible character of $G'$. Since $|G':G''|$ is prime, we deduce that $\lambda^{G'}\in \Irr(G')$. Now, we have that $\Q(\lambda^{G'})\subseteq \Q(\lambda)\subseteq \Q_{p}\subseteq \Q_{rp}$ and hence the result follows.
\end{proof}
\end{lem}

\begin{lem}\label{casoD18}
Let $G$ be a group such that $G''\not=1$, $G''$ is $p$-elementary abelian, $G/G''=\mathsf{D}_{18}$ and $p\not=3$. If $f(G)\leq 3$, then $k(G)\leq 15$. Moreover, if  $p=2$, then $k(G)\leq 10$ and if   $p$ is an odd prime with $p\equiv -1 \pmod 3$, then $k(G)\leq 12$. 
\begin{proof}
We claim that $\Q(\chi_{G'})\subseteq \Q_{3p}$  for every $\chi \in \Irr(G|G'')$. Let $\lambda \in \Irr(G'')\setminus \{1_{G''}\}$ and let $T=I_{G'}(\lambda)$. We know that $\Q(\lambda)\subseteq \Q_{p}$ and  $\lambda$ cannot be extended to an irreducible character of $G'$. Since $(|G''|,|G':G''|)=1$, applying Lemma \ref{exten}, we deduce that $\lambda$ extends to $\mu \in \Irr(T)$ with $\Q(\mu)=\Q(\lambda)\subseteq \Q_{p}$. It follows that $T<G'$ and hence we have two different possibilities. The first one is that  $T=G''$. In this case, $\lambda^{G'}\in \Irr(G')$ and hence $\Q(\lambda^{G'})\subseteq \Q(\lambda)\subseteq \Q_{p}\subseteq \Q_{3p}$. The second one is that $|T:G''|=3$. In this case, $\Irr(T/G'')=\{1,\rho, \rho^2\}$. By Gallagher's Theorem, we have that $\Irr(T|\lambda)=\{\mu, \rho\mu, \rho^2\mu\}$ and since $\Q(\rho)=\Q_{3}$, we deduce that $\Q(\psi)\subseteq \Q_{3p}$ for every $\psi \in \Irr(T|\lambda)$. Now, let  $\psi \in \Irr(T|\lambda)$. Thus, by the Clifford correspondence, $\psi^{G'}\in \Irr(G')$ and hence $\Q(\psi^{G'})\subseteq \Q(\psi)\subseteq \Q_{3p}$. Thus, $\Q(\chi_{G'})\subseteq \Q_{3p}$  for every $\chi \in \Irr(G|G'')$.

Assume that $f(G) \leq 3$. Since $\Irr(G/G'')$ contains 3 rational characters, we deduce that   $\Irr(G|G'')$ does not contain rational characters.

Assume first that $p$ is odd. By Lemma \ref{order}, we know that for every $g \in G \setminus G'$ and for every $\chi \in \Irr(G)$, $\chi(g) \in \Q_{2p}=\Q_{p}\subseteq \Q_{3p}$. Thus, by the previous claim, if  $\chi \in \Irr(G|G'')$, then $\Q(\chi)\subseteq \Q_{3p}$ and hence it  is either quadratic extension of $\Q_{3p}$ or a cubic extension of $\Q_{3p}$. We know that $\Q_{3p}$ possesses three quadratic extensions and at most one cubic extension. Thus, $|\Irr(G|G'')|\leq 3\cdot 2+1\cdot 3=9$ and hence $k(G)=|\Irr(G)|=|\Irr(G/G'')|+|\Irr(G|G'')|\leq  6+9=15$. We also observe that $\Q_{3p}$ possesses a cubic extension if and only if $p\equiv 1 \pmod 3$. Thus, if $p\equiv -1 \pmod 3$, then $k(G)\leq 12$.

Assume now that $p=2$. In this case, $\Q_{3p}=\Q_3$. By Lemma \ref{order}, we know that for every $g \in G \setminus G'$ and for every $\chi \in \Irr(G)$, $\chi(g) \in \Q_{2p}=\Q(i)$. Thus, if  $\chi \in \Irr(G|G'')$, then either $\Q(\chi)=\Q_{3}$ or $\Q(\chi)=\Q(i)$. Since $\Q(i)$ and $\Q_{3}$ are both quadratic, we have that $|\Irr(G|G'')|\leq 2\cdot 2$ and hence $k(G)\leq 6+4=10$.
\end{proof}
\end{lem}

\begin{lem}\label{casoA4}
Let $G$ be a group such that $G''\not=1$, $G''$ is $p$-elementary abelian and $G/G''=A_{4}$. If $f(G)\leq 3$,   then $k(G)\leq12$. If moreover $p\not\equiv 1 \pmod 3$,  then $k(G)\leq 9$. 
\begin{proof}
First, we study the orders of the elements of $G$. If $g \in G''$, then $o(g)$ divides $p$. If $g \in G'\setminus G''$, then $o(g)$ divides $2p$. Finally, if $g \in G \setminus G'$, then $o(g)$ divides $3p$. 

Let $\chi\in \Irr(G)$. Then, $\Q(\chi_{G''})\subseteq \Q_{p}$. If $g \in G \setminus G'$, then $\chi(g) \in \Q_{3p}$. Finally, if $g \in G'\setminus G''$, then $\chi(g)\in \Q_{2p}$. Thus, $\Q(\chi)$ is contained in $\Q_{2p}$ or in $\Q_{3p}$.

If $p=2$, then $\Q_{2p}=\Q(i)$ and $\Q_{3p}=\Q_{3}$. Therefore, we have that $k(G)=|\Irr(G)|\leq 2\cdot 2+3=7<9$.

Assume now that $p\not=2$. Then $\Q_{2p}=\Q_{p}$ and it follows that $\Q(\chi) \subseteq \Q_{3p}$ for every $\chi \in \Irr(G)$. Assume first that  $p=3$, then $\Q_{3p}=\Q_{9}$. Then $\Q_{3p}$  possesses only one quadratic extension and one cubic extension. Therefore, $k(G)=|\Irr(G)|\leq 2\cdot 1+3\cdot 1+3=8<9$. Finally, assume that $p\not=3$ is an odd prime. Then $\Q_{3p}$ has three quadratic extensions and at most one cubic extension.  It follows that $k(G)\leq 2\cdot 3+3\cdot 1+3=12$. We also have that if $p\equiv -1 \pmod 3$, then $\Q_{3p}$ has no cubic extension and hence $k(G)\leq 9$.
\end{proof}
\end{lem}

The next result completes the proof of the solvable case of Theorem A. 

\begin{thm}\label{solvable}
Let $G$ be a solvable group with $f(G)\leq 3$. Then $G \in \{\mathsf{C}_{2},\mathsf{C}_{3},\mathsf{C}_{4},\mathsf{S}_{3},\\ \mathsf{D}_{10},\mathsf{A}_{4},\mathsf{D}_{14},\mathsf{D}_{18},\mathsf{F}_{20},\mathsf{F}_{21},\mathsf{F}_{52}\}$.
\begin{proof}
 If $G$ is metabelian,  by Theorems \ref{caso2ab},\ref{caso3ab} and \ref{caso4ab}, $G\in \{\mathsf{C}_{2},\mathsf{C}_{3},\mathsf{C}_{4},\mathsf{S}_{3},\mathsf{D}_{10},\mathsf{A}_{4},\\ \mathsf{D}_{14}, \mathsf{D}_{18},\mathsf{F}_{20},\mathsf{F}_{21},\mathsf{F}_{52}\}$. Therefore, we only have to prove that $G''=1$.

Assume that $G''>1$. Taking an appropriate quotient, we may assume that $G''$ is a minimal normal subgroup of $G$. Since $G$ is solvable, we have that $G''$ is   $p$-elementary abelian for some prime $p$. We also have that $G/G''$ is a metabelian group with $f(G/G'')\leq 3$. Thus,   $G/G'' \in \{\mathsf{S}_{3}, \mathsf{D}_{10},\mathsf{A}_{4},\mathsf{D}_{14},\mathsf{D}_{18},\mathsf{F}_{20},\mathsf{F}_{21},\mathsf{F}_{52}\}$. 

We claim that we can assume that $G''$ is the unique minimal normal subgroup of $G$. Suppose that there exists $M$, a minimal normal subgroup of $G$ different of $G''$. Then $MG''/G''$ is a minimal normal subgroup of $G/G''$. On the one hand, if $G/G''\not=D_{18}$, then the only minimal normal subgroup of $G/G''$ is $G'/G''$. Thus, $G'=M\times G''$ and hence $G'$ is abelian, which is a contradiction. On the other hand, if $G/G''=D_{18}$, then the only possibility is that $|M|=3$. Let $\overline{G}=G/M$ and let $\overline{\cdot}$ denote the image in $G/M$. We have that  $f(\overline{G})\leq 3$, $\overline{G}''=\overline{G''}=MG''/M\cong G''/(M\cap G'')=G''$ and  $\overline{G}/\overline{G}'' \cong G/MG''\cong \mathsf{S}_{3}$. Therefore, $\overline{G}$ will be one of the studied cases. So, in any case, we may assume that $G$ is the only minimal subgroup of $G$, this is $G''=S(G)$. In particular, $k(G/S(G))=k(G/G'')\leq 7\leq 10$ and hence this hypothesis of Theorem \ref{Vera-Lopez2} is satisfied.

Since we are assuming that $G$ is not metabelian and $f(\mathsf{S}_4)=5>3$, we may apply Theorem \ref{Vera-Lopez3} to deduce that $\alpha(G)\geq 4$. In addition, if $k(G)\leq 11$, applying  Theorem \ref{Vera-Lopez}, we have that the only possibility is that $G''=1$, which is a contradiction. Thus, we will assume that $k(G)\geq 12$. As a consequence, if $4 \leq\alpha(G)\leq 9$, then applying Theorem \ref{Vera-Lopez2} we have that $f(G)>3$, which is impossible. Therefore, in the remaining, we will assume that $k(G)\geq 12$ and $\alpha(G)\geq 10$.

Now, we proceed to  study case by case. We study the case $G/G''=\mathsf{A}_{4}$ and the case $G/G''\not=\mathsf{A}_{4}$  separately .

\underline{Case $G/G''=\mathsf{A}_{4}$:} By Lemma  \ref{casoA4}, if $p\not\equiv 1 \pmod 3$, then $k(G)\leq 9<12$, which is imposible. Thus, we may assume that $p\equiv 1 \pmod 3$ and $k(G)=12$.  Since $\alpha(G)\geq10$, we have  that $G''$ contains a unique $G$-conjugacy class of non-trivial elements. As a consequence, $|G''|\leq 12+1=13$. We also have that $|G''|$ is a power of a prime, $p$, such that that $p\equiv 1 \pmod 3$. Thus, the only possibilities are $|G''|\in \{7,13\}$ and hence $|G|\in \{84,156\}$. By Lemma \ref{casos}, there is no group of order $84$ or $156$ with $f(G)\leq 3$ and hence we have a contradiction.

\underline{Case $G/G''\not=\mathsf{A}_{4}$:} In this case $G'/G''$ is a cyclic group. We claim that  $(|G':G''|,p)=1$.  Assume that $p$ divides  $|G':G''|$.  Then $G'$ is a $p$-group and hence $G''\subseteq \Phi(G')$. Therefore, $G'$ is cyclic and hence it is abelian, which is a contradiction. Thus,  the claim follows.  Now, we  study separately the case $G/G''=\mathsf{D}_{18}$ and the case $G/G''\in  \{\mathsf{S}_{3},\mathsf{D}_{10},\mathsf{D}_{14},\mathsf{F}_{21},\mathsf{F}_{20},\mathsf{F}_{52}\}$.

\begin{itemize}

\item \underline{Case $G/G''=\mathsf{D}_{18}$:} Since $p\not=3$, we may apply Lemma \ref{casoD18}. If $p=2$, then $k(G)\leq 10<12$ and hence we have a contradiction. Thus, we may assume that $p$ is odd. Assume now that $p$ is an odd prime such that $p\not\equiv 1 \pmod 3$. In this case $k(G)\leq 12$. Thus, $k(G)=12$ and reasoning as in the case $G/G''=\mathsf{A}_{4}$ we can deduce that $G''$ contains a unique $G$-conjugacy class of non-trivial elements.  It follows that $|G''|\leq 18+1=19$, $|G''|$ must be a power of a prime, $p$, with $p\not\equiv 1 \pmod 3$  and $|G''|=\frac{18}{|H|}+1$, where $H \leq \mathsf{D}_{18}$. Since there is no integer with the required properties,  we have a contradiction.

Assume finally that $p\equiv 1 \pmod 3$. In this case $k(G)\leq 15$. As before, we can deduce that $G''$ contains at most $4$ non-trivial conjugacy classes and hence $|G''|\leq 4 \cdot 18+1=73$. Therefore, $|G''|\in \{7, 13, 19, 31, 37, 43, 49,53, 61, 67, 73 \}$ and hence $|G| \in \{126, 234, 342, 558, 666, 774, 882, 954, 1098, 1206, 1314\}$. Applying again Lemma \ref{casos}, we have a contradiction.

\item  \underline{Case $G/G''\in  \{\mathsf{S}_{3},\mathsf{D}_{10},\mathsf{D}_{14},\mathsf{F}_{21},\mathsf{F}_{20},\mathsf{F}_{52}\}$:}  Since $(|G':G''|,p)=1$, we may apply Lemma \ref{restocasos}. Thus, if $r=|G:G'|$ and $\chi \in \Irr(G|G'')$, we have that $\Q(\chi)\subseteq \Q_{rp}$. We  study the cases $r=2,3,4$ separately.

\begin{itemize}
    \item [(i)] Case  $G/G''\in \{\mathsf{S}_{3},\mathsf{D}_{10},\mathsf{D}_{14}\}$:  In these cases  $|G:G'|=2$ and hence for all $\chi \in \Irr(G|G'')$ we have that $\Q(\chi)\subseteq \Q_{2p}=\Q_{p}$. Thus, $\Irr(G|G'')$ contains at most 5 non-rational characters. We also observe that $\Irr(G/G'')$ possesses at most $3$ non-rational character. Counting the rational characters, we have that $k(G)\leq 3+3+5=11<12$. That is a contradiction.

\item [(ii)]  Case  $G/G''=\mathsf{F}_{21}$:   If $\chi \in \Irr(G|G'')$ then $\Q(\chi)\subseteq\Q_{3p}$. Assume first that $p\not\in\{2,3\}$. Then, $\Q_{3p}$ contains three quadratic extensions and at most one cubic extension and one of these quadratic extensions is $\Q_{3}$. Since we have two characters in $\Irr(G/G'')$ whose field of values is $\Q_{3}$ there is no character in $\Irr(G|G'')$ whose field of values is $\Q_{3}$. Thus, $\Irr(G|G'')$ contains at most $2\cdot 2+3\cdot 1=7$ non-rational characters. Thus, $k(G)\leq 7+4+3=14$. Since $\Q_{3p}$ contains a cubic extension if and only if $p\equiv 1 \pmod 3$, we deduce that if $p\equiv -1 \pmod 3$, then $k(G)\leq 11<12$. Therefore, we deduce that $p\equiv 1 \pmod 3$. Now,  reasoning as in the case  $G/G''=\mathsf{D}_{18}$, we may assume that $|G''|$ contains at most $3$ non-trivial $G$-conjugacy classes. Therefore, $|G''|$ is a prime power of a prime, $p$, such that $p\equiv 1 \pmod 3$ and $|G''|-1$ must be the sum of at most three divisors of $|G/G''|=21$. It follows that  $|G''|\in \{7,43\}$. Applying that $(|G':G''|,p)=1$, we have that $|G''|=43$ and hence $|G|=21\cdot 43=903$. However, by Lemma \ref{casos}, there is no group of order $903$ with $f(G)\leq 3$.

Reasoning similarly, we can deduce that if $p=2$, then $k(G)\leq 7<12$ and hence we have a contradiction.

 Finally, assume that $p=3$. In this case $\Q_{3p}=\Q_{9}$ contains only one quadratic extension and one cubic extension. Since the unique quadratic extension of $\Q_9$ is $\Q_3$, we deduce that that $\Irr(G|G'')$ contains at most $3$ non-rational characters. Thus, $k(G)\leq 3+4+3=10<12$ and hence we have a contradiction.

\item [(iii)] Case $G/G''\in \{\mathsf{F}_{20},\mathsf{F}_{52}\}$:  Then $G/G''=\mathsf{F}_{4q}$ for $q \in \{5,13\}$. Thus, applying Lemma \ref{restocasos}, we have that  $\Q(\chi)\subseteq \Q_{4p}$ for every $\chi \in \Irr(G|G'')$. Reasoning as in the case $G/G''=\mathsf{F}_{21}$, we have that if $p\not=2$, then $\Irr(G|G'')$ contains at most $7$ non-rational characters and if $p=2$, then  $\Irr(G|G'')$ cannot contain non-rational characters. Therefore, if $p=2$ then $k(G)\leq 8<12$, which is a contradiction. Thus, we may assume that $p$ is an odd prime.

Before studying the remaining cases, we claim  that $|G''|\equiv 1 \pmod q$. Since $(|G:G''|,p)=1$, applying the Schur-Zassenhaus Theorem, we have that $G''$ is complemented in $G$ by $U\ltimes  V$, where $U$ is cyclic of order $4$ and $V$ is cyclic of order $q$. We claim that $V$ cannot fix any non-trivial element of $G''$. We have that the action of $V$ on $G''$ is coprime. Thus, by Theorem 4.34 of \cite{Isaacs}, $G''=[G'',V]\times C_{G''}(V)$. Since  $C_{G''}(V)\leq G''$ is normal in $G$ and  $G''$ is minimal normal, we have that  either $C_{G''}(V)=1$ or $C_{G''}(V)=G''$. If $C_{G''}(V)=G''$, then $G'$ is abelian, which is a contradiction. Thus, $C_{G''}(V)=1$ and hence $V$ does not fix any non-trivial element in $G''$. Therefore, $|G''|\equiv 1 \pmod q$ as we claimed.

   \begin{itemize}
        \item [a)] Case $G/G''=\mathsf{F}_{20}$: It is easy to see that $k(G)\leq 12$. If moreover, $p\not \equiv 1 \pmod 3$, then $k(G)\leq 9$, which is impossible. Thus, as in case $G/G''=\mathsf{A}_{4}$  we  may assume that $p\equiv 1 \pmod 3$  and that $G''$ possesses a unique non-trivial $G$-conjugacy class. Therefore, $|G''|\leq20+1=21$, $|G''|\equiv 1 \pmod 5$  and it is a power or a prime, $p$, $p\equiv 1 \pmod 3$. We see that there is no integer with the required properties, and hence we have a contradiction.

\item [b)] Case $G/G''=\mathsf{F}_{52}$:  It is easy to see that  $k(G)\leq 15$.   As in  case $G/G''=\mathsf{D}_{18}$, we may  assume that  $G''$ contains at most $4$ non-trivial  $G$-conjugacy classes. Therefore, $|G''|\leq 4\cdot 52+1=209$. It follows that $|G''|\equiv 1 \pmod {13}$, $|G''|\leq 209$  and it  is a  power of  a prime. Thus, $|G''|\in \{27,53,79,131,157\}$ and hence $|G|\in \{1404,2756,4108,6812,8164\}$, which contradicts Lemma \ref{casos}.
    \end{itemize}
\end{itemize}

\end{itemize}
We conclude that $G''=1$ and the result follows.
\end{proof}
\end{thm}

Now, Theorem A follows from Theorems \ref{nonsolvable} and \ref{solvable}.

\section{Further questions}

We close this paper with several possible lines of future research that, motivated by this work, have been suggested by A. Moret\'o. Theorem A shows that, as one could expect, $f(G)$ is usually much smaller than $k(G)$. It is perhaps surprising, therefore, that there could perhaps exist bounds for  $f(G)$ that are asymptotically of almost the same order of magnitude as known bounds for $k(G)$. By Brauer's  \cite{Brauer} bound, we know that the number of conjugacy classes of a finite group $G$ is $k(G)\geq\log_2\log_2|G|$. Theorem A shows that this bound does not hold if we replace $k(G)$ by $f(G)$ when $f(G)=2$ or $3$, but barely. It could be true that  $f(G)$ is at least the integer part of $\log_2\log_2|G|$.

Brauer's Problem 3 \cite{Brauer} asks for substantially better bounds for $k(G)$ in terms of $|G|$. Such a bound was obtained by Pyber \cite{Pyber}, whose bound was later improved in \cite{Keller} and \cite{BMT}. Currently the best known bound is the one given in \cite{BMT}, which  asserts that for every $\epsilon>0$, there exists $\delta>0$ such that  $k(G)>\frac{\delta \log_2|G|}{(\log_2\log_2|G|)^{3+\epsilon}}$ for all finite  groups.  Nowadays, the main open problem in this field is whether there is a logarithmic bound. More precisely, Bertram \cite{Bertram}  asked whether  $k(G)>\log_3|G|$. Theorem A shows that this bound does not hold if we replace $k(G)$ by $f(G)$, but it is not clear whether a logarithmic lower bound for $f(G)$ in terms of $|G|$ could exist. 

Another interesting problem on the number of conjugacy classes of a finite group was proposed by Bertram in \cite{Bertram}. He asked whether $k(G)\geq\omega(|G|)$, where if $n=p_1^{a_1}\dots p_t^{a_t}$ is the decomposition of the positive integer $n$ as a product of powers of pairwise different primes, $\omega(n)=a_1+\cdots+ a_n$. Theorem A also shows that this definitely does not hold if we replace $k(G)$ by $f(G)$. However, it could be true that $f(G)$ is at least the chief length of $G$ (i.e., the number of chief factors in a chief series).

\renewcommand{\abstractname}{Acknowledgements}
\begin{abstract}
This work will be part of the author’s PhD thesis, under the supervision of Alexander Moret\'o. He would like to thank him. 
\end{abstract}

\end{document}